\documentclass[draft,11pt,english]{article}

\usepackage{amsfonts,amsmath,amsthm,amscd,amssymb,latexsym,cite,verbatim,texdraw,floatflt,
caption2,pb-diagram}

\theoremstyle{plain}
\newtheorem{theorem}{Theorem}[section]

\newtheorem{corollary}[theorem]{Corollary}

\newtheorem{proposition}[theorem]{Proposition}

\newcommand{\keywords}{\textbf{Key words. }\medskip}
\newcommand{\subjclass}{\textbf{MSC 2010. }\medskip}
\renewcommand{\abstract}{\textbf{Abstract. }\medskip}

\numberwithin{equation}{section}

\newcommand{\R}{\mathbb{R}}
\newcommand{\N}{\mathbb{N}}
\newcommand{\Z}{\mathbb{Z}}
\newcommand{\Q}{\mathbb{Q}}
\newcommand{\C}{\mathbb{C}}
\newcommand{\ov}{\overline}
\newcommand{\dis}{\displaystyle}

\newcommand{\norm}[1]{\left\Vert#1\right\Vert}
\newcommand{\supp}{\textup{supp}}

\begin{document}

\title{Polynomials with integer coefficients and their zeros\thanks{Research was partially supported by the National Security Agency, and by the AT\&T Professorship}}

\author{Igor E. Pritsker\\ \\ {\em Dedicated to Professor R. M. Trigub on his 75th birthday}}



\date{}

\maketitle

\begin{abstract}
We study several related problems on polynomials with integer coefficients. This includes the integer Chebyshev problem,  and the Schur problems on means of algebraic numbers. We  also discuss interesting applications to approximation by polynomials with integer coefficients, and to the growth of coefficients for polynomials with roots located in prescribed sets. The distribution of zeros for polynomials with integer coefficients plays an important role in all of these problems.
\end{abstract}

\subjclass{11C08, 11R09, 30C15}

\keywords{Polynomials, distribution of zeros, algebraic numbers}

\section{Integer Chebyshev problem}

Let $\C_n$ and $\Z_n$ be the classes of algebraic polynomials of degree at most $n$, respectively with complex and with integer coefficients. Define the uniform norm on a compact set  $E \subset {\C}$ by
$$\norm{f}_{E} := \sup_{z \in E} |f(z)|.$$
The problem of minimizing the uniform norm on $E$ by
{\it monic} polynomials from $\C_n$ is well
known as the {\em Chebyshev problem} (see \cite{Ri90}, \cite{Ts75}, \cite{Go69}, etc.) In the classical case
$E=[-1,1]$, the explicit solution of this problem is given by the
monic Chebyshev polynomial $T_n (x) := 2^{1 -n} \cos (n \arccos x)$
of degree $n \in \N$. Using a change of variable, we extend this to an arbitrary interval $[a,b] \subset {\R}$, so that
$$t_n (x) := \left( \frac{b-a}{2} \right)^n T_n \left(
\frac{2x -a -b}{b -a} \right)$$
is a unique monic polynomial with real coefficients and the smallest uniform norm on $[a,b]$ among all {\it monic} polynomials of exact degree $n$ from $\C_n$. It is immediate that
\begin{equation} \label{1.1}
\| t_n \|_{[a,b]} = 2 \left( \frac{b-a}{4} \right)^n, \quad n
\in {\N}.
\end{equation}
Hence the {\it Chebyshev constant} for $[a,b]$ is
given by
\begin{equation} \label{1.2}
t_{\C}([a,b]) := \lim_{n \rightarrow \infty} \| t_n
\|_{[a,b]}^{1/n} = \frac{b-a}{4}.
\end{equation}
The Chebyshev problem and the Chebyshev polynomials penetrated far
beyond the original area of application in analysis,
and these ideas remain of fundamental importance, cf. \cite{Ri90}.
Many connections and generalizations were found in various areas of approximation
theory, complex analysis, special functions, etc. In particular,
the Chebyshev problem was considered on arbitrary compact sets of $\C$,
and the Chebyshev constant was identified with the transfinite
diameter and the logarithmic capacity of the set, see \cite{Go69} and
\cite{Ts75}.

A closely related problem of finding small polynomials with integer
coefficients is also classical. We say that $Q_n \in \Z_n$ is an {\it integer
Chebyshev polynomial} for $[a,b] \subset \R$ if
\begin{equation} \label{1.3}
\| Q_n \|_{[a,b]} = \inf_{0 \not\equiv P_n \in\Z_n} \| P_n \|_{[a,b]},
\end{equation}
where the $\inf$ is taken over all polynomials from $\Z_n$ that are
not identically zero. Note that $Q_n$ may
not be unique, and its degree may be less than $n.$ The {\it
integer Chebyshev constant} (or integer transfinite diameter) for
$[a,b]$ is given by
\begin{equation} \label{1.4}
t_{\Z}([a,b]) := \lim_{n \rightarrow \infty} \| Q_n \|_{[a,b]}^{1/n}.
\end{equation}
We do not require polynomials to be monic here, as integer coefficients
already provide a constraint for this extremal problem. (Requiring the
leading coefficients be monic leads to a quite different problem
considered in Borwein, Pinner and Pritsker \cite{BPP01}.)
One may readily observe that if  $b-a \geq 4$,
then $Q_n (x) \equiv 1,\ n \in {\N},$ by (\ref{1.1}) and
(\ref{1.3}), so that
\begin{equation} \label{1.5}
t_{\Z}([a,b]) = 1, \quad b-a \ge 4.
\end{equation}
On the other hand, we obtain directly from the definition and
(\ref{1.2}) that
\begin{equation} \label{1.6}
\frac{b-a}{4} = t_{\C}([a,b]) \leq t_{\Z}([a,b]),\quad b-a<4.
\end{equation}
Hilbert \cite{Hi1894} found an important upper bound
\begin{equation} \label{1.7}
t_{\Z}([a,b]) \leq \sqrt{\frac{b-a}{4}},
\end{equation}
which was originally proved in terms of the $L_2$ norm on $[a,b]$, but this
gives the same $n$th root behavior as the $L_{\infty}$ norm in (\ref{1.4}).
The asymptotic sharpness of (\ref{1.7}) was shown by Trigub \cite{Tr71}, who
observed that for the sequence of intervals $I_m:=[1/(m+4),1/m]$, we have
$$ t_{\Z}(I_m) \ge \frac{1}{m+2} \quad \mbox{and} \quad \lim_{m \to \infty} \frac{t_{\Z}(I_m)}{\sqrt{|I_m|/4}} = 1.$$
However, the exact value of $t_{\Z}([a,b])$ is not known for any segment $[a,b],\
b-a<4.$ Perhaps the most studied case is the integer Chebyshev problem
on $[0,1]$. It was initiated by Gelfond and
Schnirelman, who discovered an elegant connection with the
distribution of prime numbers (see Gelfond's comments in \cite[pp.
285--288]{Cheb44}). An exposition of related topics is found
in Montgomery \cite[Ch. 10]{Mo94}, and the later results are
contained in  Pritsker \cite{PrGSM}. Even in this best studied case, we
only know the bounds
\begin{equation} \label{1.8}
0.4213 < t_{\Z}([0,1]) < 0.42291334,
\end{equation}
where the upper bound is obtained from the definition of integer
Chebyshev constant (\ref{1.3})-(\ref{1.4}) by selecting specific
sequences of polynomials, see \cite{Fla}. The lower bound in \eqref{1.8}
in found by potential theoretic methods, cf. \cite{PrIC}. The latter
paper also contains a survey of the integer Chebyshev problem, together
with many other results and generalizations.

The integer Chebyshev problem has many applications to number theoretic
questions. In particular, we record the following connection with the growth
of the leading coefficients of polynomials.

\begin{proposition} \label{prop1.1}
Suppose that $R_m\in\Z_m$ and $P_n(x)=a_n x^n+\ldots\in\Z_n, \ a_n\neq 0,$ are two  sequences of polynomials of exact degrees $m\to\infty$ and $n\to\infty$. Assume that each $P_n$ has all zeros in $[a,b]\subset\R$, and that $P_n$ has no common zeros with $R_m$ whenever $n>m$. Then
\begin{equation} \label{1.9}
\limsup_{m \rightarrow \infty} \| R_m \|_{[a,b]}^{1/m} \, \liminf_{n \to \infty} |a_n|^{1/n} \ge 1.
\end{equation}
\end{proposition}

We use this proposition to obtain interesting information about the growth
of the leading coefficients in a problem of Schur, see the following section.

The integer Chebyshev problem \eqref{1.3}-\eqref{1.4} may be considered
as the question of uniform approximation of the constant function $f\equiv 0$
on $[a,b].$ It is crucial for the whole theory of approximation of continuous
functions by polynomials with integer coefficients. Indeed, the condition $t_{\Z}([a,b])<1$, which is equivalent to $b-a<4,$ is necessary for the
possibility of such approximation. One can find comprehensive surveys of
this old and rich area in Trigub \cite{Tr71} and Ferguson \cite{Fer80}.
The latest sharp results on the uniform approximation of functions by polynomials
with integer coefficients, which include a problem of Bernstein on
approximation of constant functions, are contained in Trigub \cite{Tr01}.

\section{Schur's problems on the growth of coefficients}

Issai Schur \cite{Sch} considered a range of problems on the relation between
the distribution of zeros and the size of coefficients for polynomials
from $\Z_n$. His work originated several important directions in analysis and
number theory. Certain number theoretic aspects of Schur's paper are discussed in detail in \cite{PrCrelle}, while \cite{PrArk} emphasizes its analytic side as
generalized by Fekete \cite{Fe23} and Szeg\H{o} \cite{Sz24}. We review and develop some of these recent results on Schur's problems from \cite{Sch}.

\subsection{Growth of the leading coefficient}

Let $E$ be a subset of the complex plane $\C.$ Consider the set of polynomials $\Z_n(E)$ with integer coefficients of the exact degree $n$ and all zeros in $E$. We denote the subset of $\Z_n(E)$ with simple zeros by $\Z_n^s(E)$. Schur showed in \cite{Sch} that the restriction of the location of zeros to $[-1,1]$ leads to the geometric growth of the leading coefficients for polynomials.

\noindent{\bf Theorem A} (Schur \cite{Sch}, Satz VII) {\em Let $P_n(x) = a_nx^n+\ldots\in\Z_n^s([-1,1])$ be an arbitrary sequence of polynomials with degrees $n\to\infty$. The infimum $L$ of $\liminf_{n\to\infty} |a_n|^{1/n}$ for all such sequences satisfies}
\begin{align} \label{2.1}
\sqrt{2} \le L \le \sqrt{1+\sqrt{2}}.
\end{align}

We use Proposition \ref{prop1.1} and results on the integer Chebyshev problem from
Section 1 to obtain the following improvement.

\begin{theorem} \label{thm2.1}
Let $P_n(x) = a_nx^n+\ldots\in\Z_n^s([-1,1]),\ n\in\N,$ be a sequence of polynomials. The infimum $L$ of $\liminf_{n\to\infty} |a_n|^{1/n}$ among all such sequences satisfies
\begin{align} \label{2.2}
1.53770952 \le L \le 1.54170092.
\end{align}
\end{theorem}

This result may be easily transformed into an analogous statement for the segment $[0,1]$, by using the change of variable $t=x^2,$  which gives a corresponding improvement for Satz X of \cite{Sch}. From a more general point of view, we obtain

\begin{theorem} \label{thm2.2}
If $P_n(x) = a_nx^n+\ldots\in\Z_n^s([a,b]),\ b-a<4,$ is any sequence of irreducible over $\Q$ polynomials, then
\begin{align} \label{2.3}
\liminf_{n\to\infty} |a_n|^{1/n} \ge \frac{1}{t_{\Z}([a,b])} \ge \frac{2}{\sqrt{b-a}}.
\end{align}
\end{theorem}
The assumption of irreducibility may be relaxed, but one needs an assumption that depends upon the asymptotic structure of the integer Chebyshev polynomials on $[a,b].$

\subsection{Means of zeros}

Given $M>0$, we write $P_n=a_nz^n + \ldots\in\Z_n^s(E,M)$ if $|a_n|\le M$ and $P_n\in\Z_n^s(E)$ (respectively $P_n\in\Z_n(E,M)$ if $|a_n|\le M$ and $P_n\in\Z_n(E)$). Schur \cite{Sch}, \S 4-8, studied the limit behavior of the arithmetic means of zeros for polynomials from $\Z_n^s(E,M)$ as $n\to\infty,$ where $M>0$ is an arbitrary fixed number. Two of his main results in this direction are stated below. Let $D:=\{z\in\C:|z|\le 1\}$ be the closed unit disk, and let $\R_+:=(0,\infty),$ where $\R$ is the real line. For a polynomial  $P_n(z)=a_n\prod_{k=1}^n (z-\alpha_{k,n})$, define the arithmetic mean of its
zeros by $A_n:=\sum_{k=1}^n \alpha_{k,n}/n.$

\noindent{\bf Theorem B} (Schur \cite{Sch}, Satz XI) {\em If $P_n\in\Z_n^s(\R_+,M)$ is any sequence of polynomials with degrees $n\to\infty$, then}
\begin{align} \label{2.4}
\liminf_{n\to\infty} A_n \ge \sqrt{e} > 1.6487.
\end{align}

\medskip
\noindent{\bf Theorem C} (Schur \cite{Sch}, Satz XIII) {\em If $P_n\in\Z_n^s(D,M)$ is any sequence of polynomials with degrees $n\to\infty$, then}
\begin{align} \label{2.5}
\limsup_{n\to\infty} |A_n| \le 1-\sqrt{e}/2 < 0.1757.
\end{align}

Schur remarked that the $\limsup$ in \eqref{2.5} is equal to $0$ for {\em monic} polynomials from $\Z_n(D)$ by Kronecker's theorem \cite{Kr1857}. We proved \cite{PrCR} that $\lim_{n\to\infty} A_n = 0$ for any sequence of polynomials from Schur's class $\Z_n^s(D,M),\ n\in\N.$ This result is obtained as a consequence of the asymptotic equidistribution of zeros near the unit circle. Namely, if $\{\alpha_{k,n}\}_{k=1}^n$ are the zeros of $P_n$, we define the zero counting measure
\[
\tau_n := \frac{1}{n} \sum_{k=1}^n \delta_{\alpha_{k,n}},
\]
where $\delta_{\alpha_{k,n}}$ is the unit point mass at $\alpha_{k,n}$. Consider the normalized arclength measure $\mu_D$ on the unit circumference, with $d\mu_D(e^{it}):=\frac{1}{2\pi}dt.$ If $\tau_n$ converge weakly to $\mu_D$ as $n\to\infty$ ($\tau_n \stackrel{*}{\rightarrow} \mu_D$) then
\[
\lim_{n\to\infty} A_n = \lim_{n\to\infty} \int z\,d\tau_n(z) = \int z\,d\mu_D(z) = 0.
\]
Thus Schur's problem is solved by the following result \cite{PrCR}.

\begin{theorem} \label{thm2.3}
If $P_n(z) = a_nz^n + \ldots\in\Z_n^s(D),\ n\in\N,$ satisfy
\begin{equation}\label{2.6}
\lim_{n\to\infty} |a_n|^{1/n} = 1,
\end{equation}
then $\tau_n \stackrel{*}{\rightarrow} \mu_D$ as $n\to\infty,$ and $\lim_{n\to\infty} A_n = 0$.
\end{theorem}

Since the elementary symmetric functions in the roots of the polynomial $P_n(z) = a_{n,n}\prod_{k=1}^n (z-\alpha_{k,n}) = \sum_{k=0}^n a_{k,n} z^k$ are directly expressed through the coefficients by
\begin{equation}\label{2.7}
\sigma_m := \sum_{j_1<j_2<\ldots<j_m} \alpha_{j_1,n} \alpha_{j_2,n} \ldots  \alpha_{j_m,n} = (-1)^m\, \frac{a_{n-m,n}}{a_{n,n}},
\end{equation}
we realize that Schur's result of Theorem B may be interpreted as an attempt of giving a sharp lower bound for the growth of $|a_{n-1,n}|$ when $P_n\in\Z_n^s(\R_+,M).$ Similarly, Theorem C establishes a limitation on the growth of $|a_{n-1,n}|$ for $P_n\in\Z_n^s(D,M).$ Thus we arrive at a very interesting problem of finding the rates of the fastest (or the slowest) asymptotic growth for the coefficients with $n$. We first state an immediate consequence of Theorem \ref{thm2.3}.

\begin{corollary} \label{cor2.4}
If $P_n(z) =\sum_{k=0}^n a_{k,n} z^k\,\in\Z_n^s(D,M),\ n\in\N,$ then
\[
\lim_{n\to\infty} \frac{|a_{n-1,n}|}{n} = 0.
\]
\end{corollary}
Hence $|a_{n-1,n}|$ can grow at most sublinearly with $n$. In fact, we can give a more precise estimate by using Corollary 1.6 of \cite{PrCR} (or Corollary 3.2 of \cite{PrCrelle}).

\begin{theorem} \label{thm2.5}
If $P_n(z) =\sum_{k=0}^n a_{k,n} z^k\,\in\Z_n^s(D,M),\ n\in\N,$ then
\begin{equation} \label{2.8}
|a_{n-1,n}| \le 8 M \sqrt{n\log{n}},\quad n\ge \max(M,55).
\end{equation}
\end{theorem}
It is interesting to note that \eqref{2.8} is sharp up to the factor $\log{n}.$
Let $p_m$ be the $m$th prime number in the increasing ordering of primes. Define the monic polynomials
\[
Q_n(z):=\prod_{m=1}^{k} \frac{z^{p_m}-1}{z-1}, \quad k\in\N,
\]
and note that each $Q_n$ has simple zeros $\{z_{j,n}\}_{j=1}^n$ at the roots of unity, and integer coefficients. Using number theoretic arguments, we show in Remark 2.8 of \cite{PrArk} that the degree of $Q_n$ is
\begin{align*}
n = \sum_{m=1}^k p_m - k = \frac{k^2 \log k}{2} + o(k^2 \log k)\quad \mbox{as } k\to\infty.
\end{align*}
Furthermore, since the sum of roots of each $(z^{p_m}-1)/(z-1)$ is equal to $-1,$ we obtain for the roots of $Q_n$ that
\[
|a_{n-1,n}| = \left|\sum_{j=1}^n z_{j,n}\right| = k \ge c {\sqrt{n/\log{n}}},
\]
where $c>0.$

We proceed with the asymptotic behavior of the coefficients $a_{n-m,n}$ for fixed
$m\in\N$ and $n\to\infty.$ These results are obtained by considering the
symmetric forms $\sigma_m$ from \eqref{2.7} for any fixed $m\in\N.$ It is convenient to first consider estimates of the sums of $m$th powers of roots, i.e. estimates of the symmetric forms
\[
s_m:=\sum_{k=1}^n \alpha_{k,n}^m,\quad m\in\N.
\]
It is clear that $s_1=\sigma_1.$ In general, the symmetric forms $s_m$ are related to the forms $\sigma_m$ by the well known Newton's formulas, cf. \cite[p. 78]{Pra}:
\begin{equation} \label{2.9}
m\sigma_m = \sum_{j=1}^m (-1)^{j-1} s_j \sigma_{m-j}.
\end{equation}
We first give a generalization of Corollary 1.6 \cite{PrCR}, and of Corollary 3.2 \cite{PrCrelle}).

\begin{theorem} \label{thm2.6}
If $P_n(z) = a_{n,n}\prod_{k=1}^n (z-\alpha_{k,n})\,\in\Z_n^s(D,M),\ n\in\N,$ then
\begin{equation} \label{2.10}
|s_m| \le (24m+16) \sqrt{n\log{n}},\quad n\ge\max(M,55).
\end{equation}
\end{theorem}
The above result can be stated in an equivalent form estimating the rates of convergence to zero for the arithmetic means of the $m$th powers $s_m/n.$ Combining \eqref{2.9} with \eqref{2.10}, we generalize Theorem \ref{thm2.5} as follows.

\begin{theorem} \label{thm2.7}
If $P_n(z) =\sum_{k=0}^n a_{k,n} z^k\,\in\Z_n^s(D,M),\ n\in\N,$ then
\begin{equation} \label{2.11}
|a_{n-m,n}| \le C(m,M)\,(n\log{n})^{m/2},\quad m,n\in\N,
\end{equation}
where $C(m,M)>0$ depends only on $m$ and $M$.
\end{theorem}
One should regard $m\in\N$ as a fixed number, and consider $n\to\infty$ in the above estimate. For $m=1$, \eqref{2.11} matches \eqref{2.8}, but sharpness of \eqref{2.11} as $n\to\infty$ remains an open question for $m>1$.

\medskip
We now return to the arithmetic means of zeros contained in $\R_+$, see Theorem B.
This result was developed in the following directions. If $P_n(z)=a_{n,n} \prod_{k=1}^n (z-\alpha_{k,n})$ is irreducible over integers, then $\{\alpha_{k,n}\}_{k=1}^n$ is called a complete set of conjugate algebraic numbers of degree $n$. When $a_n=1$, we refer to $\{\alpha_{k,n}\}_{k=1}^n$ as algebraic integers. If $\alpha=\alpha_{1,n}$ is one of the conjugates, then the sum of $\{\alpha_{k,n}\}_{k=1}^n$ is also called the trace $\textup{tr}(\alpha)$ of $\alpha$ over rationals. Siegel \cite{Si} improved Theorem B for totally positive algebraic integers to
\[
\liminf_{n\to\infty} A_n = \liminf_{n\to\infty} \textup{tr}(\alpha)/n > 1.7336105,
\]
by using an ingenious refinement of the arithmetic-geometric means
inequality that involves the discriminant of $\alpha_{k,n}$. Smyth
\cite{Sm2} introduced a numerical method of ``auxiliary
polynomials," which was used by many authors to obtain improvements of
the above lower bound. The original papers \cite{Sm1,Sm2} contain the bound $1.7719$. The later results include bounds $1.784109$ by Aguirre and Peral \cite{AP}, and $1.78702$ by Flammang \cite{Fla}. McKee \cite{McKee} recently designed a modification of the method that achieves the bound $1.78839$. The Schur-Siegel-Smyth trace problem \cite{Bor02} states:\\
\emph{Find the smallest limit point $\ell$ for the set of values of mean traces $A_n$ for all totally positive and real algebraic integers.}\\
It was observed by Schur \cite{Sch} (see also Siegel \cite{Si}), that $\ell \le 2$. This immediately follows by considering the Chebyshev polynomials $t_n(x):=2\cos(n\arccos((x-2)/2))$ for the segment $[0,4]$, whose zeros are symmetric about the midpoint $2$. They have integer coefficients, and $t_p(x)/(x-2)$ is irreducible for any prime $p$, giving the needed upper bound for $\ell$, cf. \cite{Sch}.

The Schur-Siegel-Smyth trace problem is probably the best known unsolved problem that originated in \cite{Sch}. As a partial result towards this problem, we gave the sharp lower bound $\liminf_{n\to\infty} A_n\ge 2$ for sets of algebraic numbers whose polynomials do not grow exponentially fast on compact sets of $\R_+$ of capacity (transfinite diameter) $1$, see Corollary 2.6 of \cite{PrCrelle}. More details and complete history of this problem may be also found in \cite{PrCrelle}. We illustrate our results by restricting attention to segments of length $4$ in $\R_+.$ For an interval $[c-2,c+2]\subset\R$, let $\Omega=\ov{\C}\setminus [c-2,c+2]$. The generalized Mahler measure of a polynomial $P_n(z) = a_n\prod_{k=1}^n (z-\alpha_{k,n})$ is defined by
\[
M(P_n) := |a_n| \prod_{\alpha_{k,n}\in\Omega} |\Phi(\alpha_{k,n})|,
\]
where $\Phi$ is the canonical conformal mapping of $\Omega$ onto $\ov{\C}\setminus D$ with $\Phi(\infty)=\infty.$

\begin{proposition} \label{prop2.8}
Let $P_n(z) = a_n\prod_{k=1}^n (z-\alpha_{k,n})\in\Z_n^s(\R_+),\ n\in\N,$ be a sequence of polynomials, and let $[c-2,c+2]\subset\R,\ c\ge 2$. If $\dis\lim_{n\to\infty} \left(M(P_n)\right)^{1/n} = 1$ then
\[
\liminf_{n\to\infty} \frac{1}{n} \sum_{k=1}^n \alpha_{k,n} \ge c \ge 2.
\]
\end{proposition}
We note that a more standard assumption $\lim_{n\to\infty} \norm{P_n}_p^{1/n} = 1$ for $L_p([c-2,c+2]),\ p\in(0,\infty],$ norm implies $\lim_{n\to\infty} \left(M(P_n)\right)^{1/n} = 1$.

Since Theorem B is equivalent to a statement on the growth of $|a_{n-1,n}|$ with $n$,
we are interested in the asymptotic behavior of the coefficients $a_{n-m,n}$ for a fixed $m\in\N$ and for $n\to\infty.$ To simplify and clarify the presentation, we restrict ourselves to monic polynomials, following Siegel \cite{Si}. Thus we assume for a moment that $P_n(z)=z^n+a_{n-1,n}z^{n-1}+\ldots+a_{0,n} \in \Z_n^s(\R_+,1).$ Observe that each $\sigma_m$ has $\binom{n}{m}$ number of products in the defining sum. Thus it is natural to consider the means $\sigma_m/\binom{n}{m}$. The arithmetic-geometric means inequality gives that
\[
\frac{\sigma_m}{\binom{n}{m}} \ge \left(\prod_{k=1}^n \alpha_{k,n} \right)^{\binom{n-1}{m-1}/\binom{n}{m}} = \left(\prod_{k=1}^n \alpha_{k,n} \right)^{m/n} = \left|a_{0,n}\right|^{m/n} \ge 1.
\]
It follows that if $P_n(z)=\dis\sum_{k=0}^n a_{k,n} z^k \in\Z_n^s(\R_+,1)$ then
\begin{align} \label{2.12}
\liminf_{n\to\infty} \frac{|a_{n-m,n}|}{\binom{n}{m}} = \liminf_{n\to\infty} \frac{\sigma_m}{\binom{n}{m}} \ge 1,
\end{align}
where we assume that $m\in\N$ is fixed. One can verify that the rate of growth $O\left(\binom{n}{m}\right)$ as $n\to\infty$ for $|a_{n-m,n}|$ is accurate by using the Chebyshev polynomials $t_n(x)=2\cos(n\arccos((x-2)/2))$ for the segment $[0,4]$. We now propose a generalization of the Schur-Siegel-Smyth trace problem:\\
{\em Find the sharp lower bound $\ell_m,\ 1\le m\le n-1,$  for the $\liminf$ in \eqref{2.12} among all sequences of polynomials $P_n\in\Z_n^s(\R_+,1)$.}\\
Clearly, we have that $\ell_1=\ell$. It is possible to obtain a precise relation between $\ell_m,\ m\ge 2,$ and $\ell$, but that argument falls outside the scope of the present paper. Instead, we show here that one can make an even stronger conclusion under the assumptions of Proposition \ref{prop2.8}.

\begin{theorem} \label{thm2.9}
Let $P_n(z) = a_n\prod_{k=1}^n (z-\alpha_{k,n})\in\Z_n^s(\R_+),\ n\in\N,$ be a sequence of polynomials, and let $[c-2,c+2]\subset\R,\ c\ge 2$. If $\dis\lim_{n\to\infty} \left(M(P_n)\right)^{1/n} = 1$ then
\begin{align*}
\liminf_{n\to\infty} \frac{|a_{n-m,n}|}{\binom{n}{m}} \ge c^m \ge 2^m, \quad m\in\N.
\end{align*}
\end{theorem}

It is also possible to prove similar results on the growth of coefficients for polynomials with integer coefficients and roots in sectors of the form $\{z\in\C:|\textup{Arg}\,z|\le \gamma\},$ where $\gamma<\pi/2.$

\section{Proofs}

\begin{proof}[Proof of Proposition \ref{prop1.1}]

Let $P_n(z):=a_n\prod_{j=1}^n (z-z_j)$. Recall that the resultant of
$R_m$ and $P_n$ is expressed as $R(P_n,R_m)=a_n^m \prod_{j=1}^n
R_m(z_j)$, see \cite[p. 22]{Pra}. It is clear that $R(P_n,R_m)\neq 0$ for $n>m,$
as $R_m(z_j)\neq 0,\ j=1,\ldots,n,$ by our assumption. Furthermore, this resultant is an integer,
because it has a determinant representation in terms of the coefficients of $R_m$
and $P_n$, cf. \cite[p. 21]{Pra}. Hence
\[
|a_n|^m\norm{R_m}_{[a,b]}^n \ge \left|a_n^m\prod_{j=1}^n R_m(z_j)\right| \ge 1, \quad n>m.
\]
Thus \eqref{1.9} follows by first raising the above inequality to the power $1/(mn)$, and then letting $n\to\infty$ and $m\to\infty$ in order.

\end{proof}

\begin{proof}[Proof of Theorem \ref{thm2.1}]

Let $P_n(x) = a_nx^n+\ldots\in\Z_n^s([-1,1])$ be an arbitrary sequence of polynomials. We use a sequence of polynomials $R_m$ with small sup norms on $[-1,1]$ in Proposition \ref{prop1.1} to prove the lower bound in \eqref{2.2}. In fact, these small polynomials have the following form
\[
R_m(x) = \prod_{i=1}^K Q_{m_i,i}^{[ s_i m]}(x), \quad m\in\N,
\]
where $Q_{m_i,i}(x) = b_{m_i} x^{m_i}+\ldots \in \Z_{m_i}$ are irreducible over $\Q$ polynomials, and  $0<s_i<1,\ i=1,\dots,K,$ with $\sum_{i=1}^K s_i m_i =1$. Such polynomials are used to obtain virtually all upper bounds for the integer Chebyshev constant, see \cite{PrIC}, and they were also used to obtain the upper bound in \eqref{1.8}, cf. \cite{Fla}. Since we deal with $[-1,1]$ instead of $[0,1]$, we need to apply the change of variable $t=x^2$ to translate the latter upper bound into
\[
t_{\Z}([-1,1]) \le \lim_{m\to\infty} \norm{R_m}_{[-1,1]}^{1/m} < \sqrt{0.42291334} < 0.65031788,
\]
see \cite{Tr71} and \cite{PrIC} for details. In fact, many of the irreducible factors $Q_{m_i,i}$ have all roots in $[-1,1]$, hence they may also occur as factors in the polynomials $P_n$. But this obstacle on our path to the application of Proposition \ref{prop1.1} may be easily removed. Note that we may only have at most $K$ irreducible common factors in the sequences $R_m$ and $P_n$, and that these factors may only occur in $P_n$ once, as each $P_n$ has simple roots. Thus we can drop all possible factors $Q_{m_i,i}$ in $P_n$, and obtain a new sequence of polynomials $\tilde P_n$ that have no common zeros with $R_m$ for any choice of $m$ and $n$. Furthermore, the leading coefficients of $\tilde P_n$ satisfy $ |a_n| / \prod_{i=1}^K |b_{m_i}| \le |\tilde a_n| \le |a_n|.$ Hence
\[
\liminf_{n\to\infty} |a_n|^{1/n} = \liminf_{n\to\infty} |\tilde a_n|^{1/n} \ge \lim_{m\to\infty} \norm{R_m}_{[-1,1]}^{-1/m} > 1.53770952,
\]
by Proposition \ref{prop1.1}, so that the lower bound in \eqref{2.2} is proved.

In order to obtain the upper bound in \eqref{2.2}, we need to exhibit a sequence of polynomials $P_n(x) = a_nx^n+\ldots\in\Z_n^s([-1,1])$ such that
\[
\liminf_{n\to\infty} |a_n|^{1/n} \le 1.54170092.
\]
The required sequence is given by the polynomials $G_n(x^2) = c_n x^{2n}+\ldots,\ n\in\N,$ where $G_n$ are the Gorshkov polynomials introduced in \cite{Gor59}. Further discussion of their properties may be found in \cite{Mo94}. In particular, Theorem 3 (due to Gorshkov) in \cite[p. 187]{Mo94} states that $\lim_{n\to\infty} |c_n|^{1/n} \approx 2.3768417062639.$ Our bound follows by taking the square root.

\end{proof}

\begin{proof}[Proof of Theorem \ref{thm2.2}]

Since $P_n$ is an irreducible polynomial of degree $n>m$, it cannot have common roots with any polynomial $R_m\in\Z_m$. In particular, we select $R_m$ as an integer Chebyshev polynomial of degree $m$, and apply Proposition \ref{prop1.1} to obtain that
\[
t_{\Z}([a,b]) \, \liminf_{n \to \infty} |a_n|^{1/n} = \limsup_{m \rightarrow \infty} \| R_m \|_{[a,b]}^{1/m} \, \liminf_{n \to \infty} |a_n|^{1/n} \ge 1.
\]
Hence the first inequality in \eqref{2.3} follows. The second inequality is a consequence of the upper estimate \eqref{1.7} for $t_{\Z}([a,b]).$

\end{proof}

\begin{proof}[Proof of Theorem \ref{thm2.3}]

We give a sketch of proof from \cite{PrCR} here. More details may be found in \cite{PrCrelle}. If $P_n(z)=a_n \prod_{k=1}^n (z-\alpha_{k,n})$ then the discriminant of $P_n$ is given by
\[
\Delta(P_n):=a_n^{2n-2} \prod_{1\le j<k\le n} (\alpha_{j,n}-\alpha_{k,n})^2.
\]
Observe that $\Delta(P_n)$ is an integer, as a symmetric form with integer coefficients in the zeros of $P_n$. Indeed, it may be written as a polynomial in the
elementary symmetric functions of $\alpha_{k,n}$, with integer
coefficients, by the fundamental theorem on symmetric forms. Since $P_n$ has simple roots, we have that $\Delta(P_n)\neq 0$ and $|\Delta(P_n)|\ge 1.$ Using the weak* compactness of the probability measures on $D$, we assume that $\tau_n \stackrel{*}{\rightarrow} \tau,$ where $\tau$ is a probability measure on $D$. Let $K_M(x,t) := \min\left(-\log{|x-t|},M\right).$ Since $\tau_n\times\tau_n \stackrel{*}{\rightarrow} \tau\times\tau,$ we obtain for the logarithmic energy of $\tau$ \cite{Ts75} that
\begin{align*}
I[\tau] &:=\iint \log\frac{1}{|x-t|}\,d\tau(x)\,d\tau(t) \\ &=
\lim_{M\to\infty} \left( \lim_{n\to\infty} \iint K_M(x,t)\,
d\tau_n(x)\,d\tau_n(t) \right) \\ &= \lim_{M\to\infty} \left(
\lim_{n\to\infty} \left( \frac{1}{n^2} \sum_{j\neq k}
K_M(\alpha_j,\alpha_k) +\frac{M}{n} \right) \right) \\ &\le
\lim_{M\to\infty} \left( \liminf_{n\to\infty} \frac{1}{n^2}
\sum_{j\neq k} \log\frac{1}{|\alpha_j-\alpha_k|} \right) \\ &=
\liminf_{n\to\infty} \frac{1}{n^2}
\log\frac{|a_n|^{2n-2}}{\Delta(P_n)} \le \liminf_{n\to\infty} \frac{1}{n^2}\log|a_n|^{2n-2} = 0.
\end{align*}
Thus $I[\tau]\le 0$. But $I[\nu]>0$ for any probability measure $\nu$ on $D$, except for $\mu_D$ \cite{Ts75}. Hence $\tau=\mu_D.$

It only remains to select the function $f(z)=z,\ z\in D$, and extend it continuously to $\C$ so that $f$ has compact support. The definition of the weak* convergence immediately gives
\[
\lim_{n\to\infty} A_n = \lim_{n\to\infty} \int z\,d\tau_n(z) = \int z\,d\mu_D(z) = 0.
\]

\end{proof}

\begin{proof}[Proof of Theorem \ref{thm2.5}]

Corollary 3.2 of \cite{PrCrelle} states that for any polynomial $P_n(z) = a_{n,n} \prod_{k=1}^n (z-\alpha_{k,n})\in\Z_n^s(D,M)$, we have
\[
\left|\frac{1}{n}\sum_{k=1}^n \alpha_{k,n}\right| \le 8 \sqrt{\frac{\log{n}}{n}},\quad n\ge \max(M,55).
\]
Since $|a_{n-1,n}|=|a_{n,n}\,\sigma_1|\le M |\sigma_1|$ by \eqref{2.7}, we obtain \eqref{2.8} as a combination of the above estimates.

\end{proof}

\begin{proof}[Proof of Theorem \ref{thm2.6}]

Since the zeros of $P_n$ come is complex conjugate pairs, we note that each $s_m$ is real, and we have that
\[
\frac{s_m}{n} = \int z^m\,d\tau_n(z) = \int \Re(z^m)\,d\tau_n(z).
\]
We state Theorem 3.1 of \cite{PrCrelle} for convenience. Let $\phi:\C\to\R$ satisfy $|\phi(z)-\phi(t)|\le A|z-t|,\ z,t\in\C,$ and $\supp(\phi)\subset\{z:|z|\le R\}.$ If $P_n(z) = a_{n,n}\prod_{k=1}^n (z-\alpha_{k,n})$ is a polynomial with integer coefficients and simple zeros, then
\begin{align} \label{3.1}
\left|\frac{1}{n}\sum_{k=1}^n \phi(\alpha_{k,n}) - \int\phi\,d\mu_D\right| \le A(2R+1) \sqrt{\frac{\log\max(n,M(P_n))}{n}}
\end{align}
for all $n\ge 55,$ where $M(P_n)=|a_{n,n}| \prod_{k=1}^n \max(|\alpha_{k,n}|,1)$ is the Mahler measure of $P_n.$ Note that $M(P_n)=|a_{n,n}|\le M$ in our case. If we let
\[
\phi(z):=\left\{
           \begin{array}{ll}
             \Re(z^m), & |z|\le 1, \\
             \Re(z^m)(m+1-m|z|), & 1\le |z|\le 1+1/m, \\
             0, & |z|\ge 1+1/m,
           \end{array}
         \right.
\]
then
\[
\frac{s_m}{n} = \int \phi(z)\,d\tau_n(z) \quad \mbox{and} \quad \int\phi\,d\mu_D = 0.
\]
An elementary computation shows that the partial derivatives $\phi_x$ and $\phi_y$ exist on $\C\setminus S,$ where $S:=\{z:|z|=1 \mbox{ or } |z|=1+1/m\},$ and they are given by
\begin{align*}
\phi_x(z)=\left\{
           \begin{array}{ll}
             \Re(mz^{m-1}), & |z|<1, \\
             \Re(mz^{m-1})(m+1-m|z|) - \dis\frac{mx\, \Re(z^m)}{\sqrt{x^2+y^2}} , & 1<|z|<1+\frac{1}{m}, \\
             0, & |z|>1+\frac{1}{m},
           \end{array}
         \right.
\\ \\
\phi_y(z)=\left\{
           \begin{array}{ll}
             \Re(miz^{m-1}), & |z|<1, \\
             \Re(miz^{m-1})(m+1-m|z|) - \dis\frac{my\, \Re(z^m)}{\sqrt{x^2+y^2}} , & 1<|z|<1+\frac{1}{m}, \\
             0, & |z|>1+\frac{1}{m},
           \end{array}
         \right.
\end{align*}
where we let $z=x+iy.$ Furthermore, we obtain that $|\phi_x(z)|\le 2me$ and $|\phi_y(z)|\le 2me$ for $z=x+iy\in\C\setminus S.$ The Mean Value Theorem now gives
\[
|\phi(z)-\phi(t)|\le |z-t|\, \sup_{\C\setminus S} \sqrt{\phi_x^2+\phi_y^2} \le 2me\sqrt{2}\, |z-t| < 8m\, |z-t|.
\]
Hence we use \eqref{3.1} with $A=8m$ and $R=1+1/m$ to obtain \eqref{2.10}.

\end{proof}

\begin{proof}[Proof of Theorem \ref{thm2.7}]

The recurrence relation \eqref{2.9} immediately implies the estimate
\[
|\sigma_m| \le \frac{1}{m} \sum_{j=1}^m |s_j| |\sigma_{m-j}|.
\]
Hence \eqref{2.11} is obtained from \eqref{2.10} by a standard inductive argument in $m\in\N.$

\end{proof}

\begin{proof}[Proof of Proposition \ref{prop2.8}]

Theorem 2.1 of \cite{PrCrelle} guarantees that the counting measures $\tau_n$ in the zeros of $P_n$ converge to the equilibrium measure $\mu$ of $[c-2,c+2]$ in the weak* topology. It is known \cite{Ts75} that
\begin{align} \label{3.2}
d\mu(x) = \frac{dx}{\pi\sqrt{4-(x-c)^2}},\quad x\in(c-2,c+2).
\end{align}
Thus we have that
\[
\liminf_{n\to\infty} \frac{1}{n} \sum_{k=1}^n \alpha_{k,n} = \liminf_{n\to\infty} \int x\,d\tau_n(x) \ge \int_{c-2}^{c+2} \frac{x\,dx}{\pi\sqrt{4-(x-c)^2)}} = c.
\]

\end{proof}

\begin{proof}[Proof of Theorem \ref{thm2.9}]

It follows from \eqref{2.7} that
\begin{align*}
&\liminf_{n\to\infty} \frac{|a_{n-m,n}|}{\binom{n}{m}} \ge m!\, \liminf_{n\to\infty} \frac{\sigma_m}{n^m} \\ &= m!\, \liminf_{n\to\infty} \int\ldots\int_{x_1<\ldots<x_m} x_1 \ldots x_m\,d\tau_n(x_1)\ldots d\tau_n(x_m).
\end{align*}
We again use Theorem 2.1 of \cite{PrCrelle} to conclude that $\tau_n \stackrel{*}{\rightarrow} \mu$, where $\mu$ is defined by \eqref{3.2}. Hence
\begin{align*}
\liminf_{n\to\infty} \frac{|a_{n-m,n}|}{\binom{n}{m}} \ge m! \int\ldots\int_{x_1\le \ldots\le x_m} x_1 \ldots x_m\,d\mu(x_1)\ldots d\mu(x_m).
\end{align*}
The latter integral is evaluated by using the symmetry of the integrand in the variables $x_1, \ldots, x_m,$ and by passing from the integral over a simplex to an integral over a cube, which gives
\begin{align*}
&m! \int\ldots\int_{c-2\le x_1\le \ldots\le x_m\le c+2} x_1 \ldots x_m\,d\mu(x_1)\ldots d\mu(x_m) \\ &= \int_{c-2}^{c+2}\ldots\int_{c-2}^{c+2} x_1 \ldots x_m\,d\mu(x_1)\ldots d\mu(x_m) \\ &= \left(\int_{c-2}^{c+2} x\,d\mu(x)\right)^m = c^m.
\end{align*}

\end{proof}

\bigskip

CONTACT INFORMATION

\medskip
I. E. Pritsker\\
Department of Mathematics, Oklahoma State University, Stillwater, OK 74078,   U.S.A.\\
igor@math.okstate.edu
\end{document}